\documentclass[12pt,a4paper]{article}
\usepackage{graphicx} %
\usepackage{amsmath}
\usepackage{amssymb}
\usepackage{mathtools}
\usepackage{hyperref}
\usepackage{enumitem}
\usepackage{amsthm}
\usepackage[dvipsnames]{xcolor}
\hypersetup{
  colorlinks=true,
  citecolor=blue,
  linkcolor=purple!80
}

\usepackage{float}
\usepackage{bm}
\usepackage{natbib}
\usepackage[nameinlink,noabbrev,capitalize]{cleveref}

\newcommand\cJ{{\cal J}}
\newcommand\cC{{\cal C}}
\newcommand\cR{{\cal R}}
\newcommand\cA{{\cal A}}
\newcommand\cM{{\cal M}}

\newcommand\R{{\mathbb R}}
\newcommand\N{{\mathbb N}}

\def\lM{{\cal M}}
\def\lK{{\cal K}}
\def\div{\mathrm{div}}
\def\mx{{\bf x}}

\def\vu{{\mathsf u}}
\def\vp{{\mathsf p}}
\def\bA{{\mathbf{A}}}
\def\bM{{\mathbf{M}}}
\def\Id{{\mathbf{I}}}
\def\Hmo#1{{{\rm H}^{-1}(#1)}}
\def\Hzo#1{{{\rm H}^{1}_0(#1)}}
\def\Mab{\lM(\alpha,\beta;\Omega)}
\def\mxi{{\bm \xi}}
\def\Linf#1{{{\rm L}^\infty(#1)}}
\def\Lone#1{{{\rm L}^1(#1)}}
\def\Lp#1#2{{{\rm L}^{#1}(#2)}}
\def\Sym{{\rm Sym}_d}
\def\us{{\vu^\ast}}
\def\ps{{\vp^\ast}}
\def\As{{\bA^{\!\ast}}}
\def\Ms{{\bM^{\ast}}}
\def\ths{{\theta^\ast}}
\def\dA{{\delta\! \bA}}
\def\dth{{\delta \theta}}
\def\pL#1#2{{{\rm L}^{#1}(#2)}}
\def\varg{{\mathrm{VaR}_\gamma}}
\def\cvarg{{\mathrm{CVaR}_\gamma}}

\newcommand{\weakstar}{\stackrel{\ast}\rightharpoonup}
\newcommand{\weakH}{\stackrel{H}\rightharpoonup}

\usepackage{algorithm}
\usepackage{algorithmic}

\newtheorem{theorem}{Theorem}[section]
\makeatletter
\let\c@algorithm\c@theorem
\makeatother

\newtheorem{lemma}[theorem]{Lemma}

\newtheorem{remark}[theorem]{Remark}
\newtheorem{definition}[theorem]{Definition}

\newtheorem{example}[theorem]{Example}

\newtheorem{assumption}[theorem]{Assumption}
\crefname{assumption}{Assumption}{Assumptions}

\title{Optimal design with uncertainties: a risk-averse approach}
\author{Amal Alphonse\footnote{Weierstrass Institute, Berlin, Germany,  
\texttt{alphonse@wias-berlin.de}}\and 
Petar Kun\v{s}tek\footnote{Department of Mathematics, Faculty of Science,  University of Zagreb, Croatia, 
\texttt{petar@math.hr}} \and 
Marko Vrdoljak\footnote{Department of Mathematics, Faculty of Science,  University of Zagreb, Croatia, 
\texttt{marko@math.hr}}}

\begin{document}

\maketitle

\begin{abstract}
We study a class of stochastic optimal design problems for elliptic partial differential equations in divergence form, where the coefficients represent mixtures 
of two conducting materials. The objective is to minimize a generalized risk measure of the system response, incorporating uncertainty 
in the loading through probability distributions. We establish existence of relaxed optimal designs via homogenization theory and derive first-order stationarity conditions satisfied by the optima. 
 Based on these %
 necessary %
 conditions, we develop an optimality criteria algorithm 
for numerical computations. The stochastic component is treated using a truncated Karhunen--Loève expansion, allowing evaluation 
of the value-at-risk  (VaR) and conditional value-at-risk (CVaR) contributions arising from the sensitivity analysis and featured in the algorithm. The method 
 is illustrated for an example involving CVaR-based compliance minimization.
\end{abstract}

\section{Introduction}

Optimal design of structures and materials governed by partial differential equations has been extensively studied in the 
deterministic setting, where uncertainties in loading, boundary conditions, and material properties are neglected \cite{MT, BendKik, A, BS03, SM13}. 
In many applications, 
however, such uncertainties have a significant influence on the system response, and deterministic optimization may lead to designs 
that perform poorly or even fail under random perturbations. This has partly motivated the development of stochastic and risk-averse 
PDE-constrained optimization, in which random inputs are explicitly modeled and optimization criteria are formulated to control 
statistical properties of the resulting random state \cite{CHPRS11, AD15, DHP15, FP18}.

A foundational framework for risk-averse PDE-constrained optimization was established by Kouri and Surowiec in \cite{KS18, Corrigendum} (see also \cite{MR4145243}). 
They studied optimization problems 
constrained by elliptic PDEs with random forcing and incorporated coherent risk measures, most notably the Conditional 
Value-at-Risk (CVaR) to quantify the risk of high-impact, low-probability outcomes. Their analysis included the existence of optimal 
controls %
and the derivation of  first-order optimality conditions. 

The present work is motivated by extending these ideas to optimal design problems arising in composite materials, 
where the control variable is the spatial distribution of two phases with different conductivity properties. 
Incorporating uncertainty into such design problems presents new mathematical challenges. The response of the system to random 
loads becomes a random field, and the objective is typically a functional of the entire distribution, not merely its expectation. 
As highlighted in %
e.g., \cite{SDR14, MR4926315} and references therein, coherent risk measures (such 
as CVaR) provide a robust alternative to expected-value objectives by penalizing rare but severe realizations. However, the presence of risk measures induces nonsmoothness in the objective functional, which in turn requires careful mathematical analysis.

Let us present the problem of interest in this work. The aim is to find the best arrangement of given materials in an open and bounded domain $\Omega\subseteq\R^d$ %
so that the resulting body has optimal properties regarding $m$ different regimes. We consider the simplest case of two isotropic materials possessing conductivities $\alpha$ and $\beta$ with $0<\alpha<\beta$. If we let $\Omega_\alpha\subseteq\Omega$ be the measurable set occupied by the first material, and $\chi \equiv \chi_{\Omega_\alpha}$ its characteristic function, then the conductivity can be written as
\begin{equation}\label{Achar}
\bA=\chi\alpha\Id +(1-\chi)\beta\Id
\end{equation}
where $\Id$ is the identity. Let  $(S, \mathcal{F},\mathbb{P})$  be a probability space, and suppose that $f_i\in \Lp p{S; \Hmo\Omega}$ are given for some $p\in (1,\infty)$, %
and introduce the boundary value problems (on a pointwise a.s. level)
\begin{equation}\label{ses}
\left\{
\begin{array}{l}
-\div(\bA\nabla u_i)=f_i(s,\cdot)\\
u_i\in \Hzo\Omega\,
\\
\end{array}\,
\qquad i=1, \ldots m\,.
\right.
\end{equation}
For example, $u_i$ could model the electrical potential or temperature for the $i\textsuperscript{th}$ state. In general, these equations allow us to take into account $m$ different states and reflect the fact that a device is often required to be robust in multiple scenarios.
Define a quantity of interest $\widehat\cJ\colon \Linf{\Omega; \{0,1\}}  \to L^0(S; \mathbb{R})$ (where by $L^0$ we mean the set of measurable functions) that measures the efficacy of the mixture by 
\[
\widehat\cJ(\chi)(s)=\int_\Omega \Bigl(\chi(\mx)\, g_\alpha(s,\mx,\vu(s,\mx)) + (1-\chi(\mx))\, g_\beta(s,\mx,\vu(s,\mx))\Bigr)\, d\mx\,
\]
where $\vu=(u_1,\ldots,u_m)$ with the $u_i$ determined by \eqref{ses}, $\bA$ is given by \eqref{Achar} and $g_\alpha$ and $g_\beta$ are sufficiently smooth functions satisfying certain growth conditions, see \cref{rem:g_functions} for some examples. Although $g_\alpha$ and $g_\beta$ could depend on $u_i$ in a more complicated fashion (such as depending on $\nabla u_i$) and such cases are also interesting in applications, in this work, we consider only dependence on the pointwise state values.

The functional is considered in this form in \cite{MT} for the single-state problem and in \cite{A} for the multi-state case, as it encompasses the most general class under consideration, while also allowing for a possible dependence of the functional on the random parameter $s$.

In order to treat the fact that the objective function is probabilistic, we can consider a risk-averse %
optimal design problem. For a given risk measure $\cR\colon \Lone S \to \mathbb{R}$, %
the aim is to
find a characteristic function $\chi$ on $\Omega$ that minimizes $\widehat \cJ$ in a risk-averse sense:
$\min_{\substack{\chi \in \Linf{\Omega; \{0,1\}}}} %
\cR[\widehat \cJ(\chi)]$,
possibly supplemented by some regularization term.

As is standard in deterministic optimal design, in order to avoid ill posedness of the problem (see e.g. \cite[\S 3.1.5]{A}), a proper relaxation of the problem is required, leading to the use of composite materials formed as fine mixtures of the original phases. Accordingly, instead of characteristic functions for the first material, its local volume fraction 
$\theta$, taking values in $[0,1]$, is introduced. However, the local volume fraction $\theta$ alone does not fully define the composite; rather a set of admissible composites which we denote by $K(\theta)$ (this contains all the effective materials with proportions $\theta$)  does, see \eqref{eq:defnOfKTheta}. Therefore, the control variable becomes $(\theta,\bA)$  with $\bA\in K(\theta)$ almost everywhere in $\Omega$. Let us denote
\begin{align}
\cA := \Big\{ (\theta, \bA) \in \Linf{\Omega; [0,1]}\times \Linf{\Omega; \Sym}  : \bA(\mx) \in {K}(\theta(\mx)) \text{ for a.e. $\mx \in \Omega$}\Big\}    \label{eq:defn_mathcal_A}
\end{align}
where the notation $\Sym$ means the set of real symmetric $d\times d$ matrices. 
Finally, the problem we study in this paper can be formulated as %
\begin{equation}
\min_{(\theta,\bA) \in \cA}\cR[\cJ(\theta, \bA)] + \varrho(\theta)   \tag{P} \label{eq:min_problem_intro}
\end{equation}
where
\begin{equation}
\cJ(\theta, \bA)(s) =\int_\Omega \Bigl(\theta(\mx)\, g_\alpha(s,\mx,\vu(s,\mx)) + (1-\theta(\mx))\, g_\beta(s,\mx,\vu(s,\mx))\Bigr)\, \mathrm{d}\mx\quad\text{a.s. in $s$},      \label{eq:cJ_intro}
\end{equation}
$\vu$ is uniquely determined by $\bA$  through  \eqref{ses}, $\cR$ is a risk measure and $\varrho$ is a regularizer. %
We will formulate our problem in full rigor in  \cref{sec:existence}. %
Let us give some examples of the objective functional, risk measure and regularizer. 
\begin{example}\label{rem:g_functions}
An example of the objective function for $m=1$ is obtained with the choice $g_\alpha(s,\mx, u) = g_\beta(s,\mx,u) = f(s,\mx)u$, whence 
\[\cJ(\theta, \bA)(s) = \int_\Omega f(s,\mx)u(s,\mx)\;\mathrm{d}\mx,\]
which is the energy dissipation functional. Another example relates to tracking-type objectives typical in optimal control where $g_\alpha(s,\mx,u) = |u(s,\mx)-u_\alpha(\mx)|^2$ and $g_\beta(s,\mx,u) = |u(s,\mx)-u_\beta(\mx)|^2$ where $u_\alpha, u_\beta$ are given (deterministic) desired states; this choice leads to 
\[\cJ(\theta, \bA)(s) = \int_\Omega (\theta(\mx)|u(s,\mx)-u_\alpha(\mx)|^2 + (1-\theta(\mx))|u(s,\mx)-u_\beta(\mx)|^2\;\mathrm{d}\mx.\]
For the multiple state problem with $m>1$, we could consider simply the sum of individual functions associated to each state:
\[\cJ(\theta, \bA)(s) = \sum_{i=1}^m \int_\Omega \Bigl(\theta(\mx)\, g_\alpha^i(s,\mx,u_i(s,\mx)) + (1-\theta(\mx))\, g_\beta^i(s,\mx,u_i(s,\mx))\Bigr)\, \mathrm{d}\mx. \]
These examples were pointed out in \cite[\S 3.1]{A}. 

Regarding the risk measure, we can take $\cR$ to be CVaR as mentioned above (see \eqref{eq:defnCVAR} for a definition) or we could work in the risk-neutral setting where $\cR = \mathbb{E}$ is the expected value, for example.

For the regularizer $\varrho$, one could consider $\varrho(\theta) = \frac12 \lVert \theta \rVert_{L^2(\Omega)}^2$ or the typical choice $\varrho(\theta) = \lambda \int_\Omega \theta\;\mathrm{d}\mx$.
\end{example}

\begin{remark}
A different approach is to pose the probability maximization problem
\begin{equation}
\max_{(\theta, \bA) \in \cA} \mathbb{P}(\cJ(\theta, \bA)(s) \leq c)\label{eq:max_problem}
\end{equation}
for a given acceptable threshold level $c>0$. This models the situation where we seek to find the design that maximizes the probability that our objective is satisfied  up to a small error. This is of course a different model, but it is not too difficult to see that it can be written in the form \eqref{eq:min_problem_intro} with the risk measure $\cR$ chosen to be $\varg$. However, our theory does not apply here because $\varg$ is not convex and convexity is enforced on $\cR$ in \cref{ass:onBothFunctionals}. Nevertheless, probability maximization problems such as \eqref{eq:max_problem} can be tackled directly and stationarity conditions can be derived, at least in the setting where the optimization is unconstrained and some structural assumptions on the objective in terms of the uncertain variable are present (such as convexity or quadraticity, the latter of which is satisfied by our example in \cref{sec:numerics}). Differentiability of the probability function has been the object of study in e.g. \cite{MR3273343, MR4081252} using the spherical radial decomposition, mostly in the finite-dimensional setting and to a lesser degree the reflexive, separable Banach space setting \cite{MR3935077}, which does not fit our non-separable, non-reflexive control space $\cA$. The work \cite{MR4755151} does consider a general Banach space setting, but utilizes the above-mentioned convexity. We believe it is possible to extend these results to our constrained optimization problem in a non-reflexive, non-separable space with a nonconvex dependence on the uncertainty (since the proofs in the above-cited papers either can be adapted or they simply hold verbatim when the objective function is $C^1$, as it is for us), yet this is non-trivial and outside the scope of this paper. Thus we leave this, and other interesting aspects such as pointwise chance constraints, for investigation in future work. 
 
\end{remark}

We conclude this section with a brief overview of the paper. In \cref{sec:existence}, we precisely introduce the set of admissible controls and state the assumptions that ensure the existence of a solution. We also outline the key elements of the homogenization theory employed. \cref{sec:optimality} is devoted to the  derivation of  first-order optimality conditions. %
Finally, \cref{sec:numerics} presents a variant of the optimality criteria algorithm tailored to the numerical solution of the problem. 
By employing the Karhunen--Loève expansion of the random right-hand sides, we obtain an 
explicit finite-dimensional representation of the stochastic component, which enables an efficient 
implementation of the algorithm and is illustrated by a numerical example. We conclude the paper in \cref{sec:conclusion} with some remarks.

\section{\texorpdfstring{$H$}{H}-convergence and existence of solutions}\label{sec:existence}

 Here and below, we embed the Sobolev space $\Hzo\Omega$ with the norm $\lVert u \rVert_{\Hzo\Omega} := \lVert \nabla u \rVert_{\pL{2}{\Omega}}$. 
 
 In connection with the problem involving $\widehat\cJ$ (wherein we note that $\chi$ influences $\bA$ from \eqref{Achar}, which in turn influences the boundary value problems \eqref{ses}), it is well known that no reasonable pair of topologies exists for which the mapping $\chi\mapsto u$ is continuous from the set of measurable characteristic functions on $\Omega$ to H$^1(\Omega)$, where $u$ denotes the unique solution of 
\begin{equation}\label{sdif}
\left\{
\begin{array}{l}
-\div(\bA\nabla u)=f\\
u\in {\rm H}^1_0(\Omega)\,,\\
\end{array}
\right.
\end{equation}
with $\bA$  given by \eqref{Achar}.
On the other hand, homogenization theory \cite{MT1, MT, A}  provides a natural topology on a class of coefficients $\bA$ for which the mapping $\bA\mapsto u$ is continuous for every $f\in\Hmo\Omega$ with $u$ equipped with the  weak H$^1$ topology.

To be more precise, with $\mathrm{M}_d(\R)$ denoting the space of real $d\times d$ square matrices, this topology is introduced on the set $\Mab$, defined \cite{MT1,MT}
\begin{align*}
\Mab := \{ &\bA \in \Linf{\Omega;{\rm M}_d(\R)} : \bA(\mx)\mxi\cdot\mxi \geq\alpha |\mxi  |^2 \text{ and }\\
&\bA(\mx)\mxi\cdot\mxi \geq {\frac1\beta}|\bA(\mx)\mxi|^2 \text{ for a.e. $\mx\in\Omega$} \text{ and for all }\mxi\in\R^d\},   
\end{align*}
where $|\cdot |$ denotes the Euclidean norm on $\R^d$.  The notion of $H$-convergence is important, which we recall now.
\begin{definition}[$H$-convergence]
We say that the matrix sequence $\{\bA_n\}\subseteq \Mab$ \emph{$H$-converges} to $\bA\subseteq \Mab$ with respect to the $H$-topology if for arbitrary $f \in \Hmo\Omega$ the solutions $u_n \in H^1_0(\Omega)$ of
\[-\div(\bA_n\nabla u_n) = f\]
 satisfy $u_n \rightharpoonup u$ in $H^1_0(\Omega)$ and $\bA_n\nabla u_n \rightharpoonup \bA\nabla u$ in $L^2(\Omega;\mathbb{R}^d)$ where  $u \in H^1_0(\Omega)$ solves
 \[-\div(\bA\nabla u) =f.\]
 We write this as $\bA_n\weakH\bA$.
\end{definition}
Crucially, this convergence introduces a topology that is known to be metrizable and compact on $\Mab$ \cite[Theorem 6]{Ta2}. 
Note also that the set of conductivities of the form \eqref{Achar} where $\chi$ is an arbitrary characteristic function belongs to $\Mab$.

If a sequence of characteristic functions $\chi_n$ on $\Omega$ weakly-$\ast$ converges to $\theta\in\Linf{\Omega;[0,1]}$ and  
$\bA_n=\chi_n\alpha\Id +(1-\chi_n)\beta\Id$ $H$-converges to $\bA\in\Mab$, we say that $\bA$ is a \textit{homogenized conductivity} of a two-phase composite material obtained by mixing $\alpha\Id$ and $\beta\Id$ in proportions $\theta$  and $1-\theta$, see \cite[\S 2.1.2]{A}.  The set of all such homogenized conductivities $\bA$ (i.e., $H$-convergent limits $\bA$ of the process described above) is denoted by $\lK_\theta$:
\begin{equation}
    \lK_\theta = \{ \bA \in \Mab : \text{$\bA$ is the $H$-limit of $\bA_n$ as described above}.\}
\end{equation}
This set can be understood as the set of all composite materials associated to $\theta$.

By the local character of $H$-convergence it follows that the set of homogenized conductivities $\lK_\theta$  can be characterized as L$^\infty$ matrix functions such that for almost every $\mx\in\Omega$ we have $\bA(\mx)\in K(\theta(\mx))$. Here, for any $\vartheta\in[0,1]$ the set $K(\vartheta)$ is some closed set of matrices. In fact, in our situation of mixing two isotropic phases, $K(\vartheta)$ is a set of symmetric matrices which can be explicitly characterized in terms of their eigenvalues (see \cite[Proposition 10]{MT} for precise details, see also \cite{LC1}, \cite[Theorem 2.2.13]{A}): 
\begin{align}
\nonumber    K(\vartheta) = \Bigg\{ \mathbf{M} \in \Sym &: \lambda_\vartheta^- \leq \lambda_i \leq \lambda_\vartheta^+ \quad \text{ for every } 1 \leq i \leq d,\\
\nonumber    &\;\;\sum_{i=1}^d \frac{1}{\lambda_i-\alpha} \leq \frac{1}{\lambda_\vartheta^- - \alpha} + \frac{d-1}{\lambda_\vartheta^+-\alpha},\\
    &\;\;\sum_{i=1}^d \frac{1}{\beta - \lambda_i} \leq \frac{1}{\beta - \lambda_\vartheta^-} + \frac{d-1}{\beta - \lambda_\vartheta^+}\Bigg\}\label{eq:defnOfKTheta}
\end{align}
where we used the notation $\lambda_i$ to mean the eigenvalues of the matrix $\mathbf{M}$ and $\lambda_\vartheta^-$ and $\lambda_\vartheta^+$ are the harmonic and arithmetic means of $\alpha$ and $\beta$, i.e., $\lambda_\vartheta^- = (\vartheta\slash \alpha + (1-\vartheta)\slash \beta)^{-1}$ and $\lambda_\vartheta^+ = \vartheta\alpha + (1-\vartheta)\beta$. %

Now it is easy to see that the set $\cA$\footnote{Note that  $\cA$ corresponds to the set $\mathcal{CD}$ defined in \cite[Equation (3.39)]{A}.} defined in \eqref{eq:defn_mathcal_A}
\begin{align*}
\cA := \Big\{ (\theta, \bA) \in \Linf{\Omega; [0,1]}\times \Linf{\Omega; \Sym}  : \bA(\mx) \in {K}(\theta(\mx)) \text{ for a.e. $\mx \in \Omega$}\Big\}    
\end{align*}
is compact with respect to the weak-$\ast$ topology for $\theta$ and the $H$-topology for $\bA$. 
For fixed $\vartheta,$ the set $K(\vartheta)$ is convex \cite[Theorem 2.2.13]{A}, but the set $\cA$ is not \cite[Example 2.1]{V2}. %

We return to the problem \eqref{eq:min_problem_intro}. To begin with, let us consider the well posedness of the constituent state equations \eqref{ses}.
\begin{lemma}
For $\bA \in \Mab$, the problem \eqref{ses} is uniquely solvable and $u_i \in L^p(S; H^1_0(\Omega))$ for $i = 1, \ldots, m$. Furthermore, we have the following a.s. estimate:
\begin{equation}
\|u_i(s,\cdot)\|_{\Hzo\Omega}\leq \alpha^{-1}\| f_i(s,\cdot)\|_{\Hmo\Omega}\,,\;i=1,\ldots, m\,.\label{eq:a_priori_u_i}
\end{equation}
\end{lemma}
\begin{proof}
    Well posedness almost surely in $s\in S$, together with the  \emph{a priori} estimate \eqref{eq:a_priori_u_i}, follows directly from the Lax--Milgram lemma. Defining the linear operator $T\colon H^1_0(\Omega) \to H^{-1}(\Omega)$ by $Tu := -\div(\bA \nabla u)$, we can write $u_i(s) = T^{-1}f_i(s)$. By \cite[Section V.5, Corollary 2]{Y95} $u_i$ is strongly measurable and, by the estimate \eqref{eq:a_priori_u_i} above,  belongs to the specified Bochner space.
\end{proof}
Since the right-hand sides $f_1,\ldots,f_m\in \pL p{S;\Hmo\Omega}$ from \eqref{ses} are considered  fixed, to simplify notation we introduce  the mapping \[\text{$G\colon\Mab\to\left[\pL p{S;\Hzo\Omega}\right]^m, \quad G(\bA) = \vu$
determined by \eqref{ses}. }\]
Recall that the (relaxed) function $\cJ\colon \cA \to \Lone S$ from \eqref{eq:cJ_intro} is defined via
\[
\cJ(\theta, \bA)(s) =\int_\Omega \Bigl(\theta(\mx)\, g_\alpha(s,\mx,\vu(s,\mx)) + (1-\theta(\mx))\, g_\beta(s,\mx,\vu(s,\mx))\Bigr)\, \mathrm{d}\mx
\]
such that $\vu = G(\bA)$, i.e.  $\vu$ and $\bA$ are related via \eqref{ses}.
Regarding the functions $g_\alpha$, $g_\beta$, we assume the following.
\begin{assumption}\label{ass:on_gs_first}
The functions 
 $g_\alpha$ and $g_\beta$ are  Carath\'eodory functions (measurable in $(s,\mx)$ on the product space $S\times\Omega$ and continuous in $\vu$) satisfying the  growth conditions:
\begin{equation}\label{growth1}
|g_\gamma(s,\mx,\vu)|\leq \varphi_\gamma(s,\mx) + \psi_\gamma(s,\mx){|\vu|}^{q_1} \qquad{\rm for}\; \gamma=\alpha,\beta,
\end{equation}
with 
$1\leq q_1\leq\min\{p,q_\ast\}$ where
\[ 
q_\ast=\left\{\begin{array}{rl}
+\infty&\,:\; d\leq 2,\\
{\frac{2d}{d-2}}&\,: \;d>2\,,
\end{array}\right.
\]
and $\varphi_\gamma\in\pL1{S\times\Omega}$, $\psi_\gamma\in\Lp{q_2}{S;\pL{q_3}{\Omega}}$ with $q_2=\displaystyle\frac{p}{p-q_1}$, and  
$q_3>\displaystyle\frac{q_*}{q_*-q_1}$. %

\end{assumption}

In addition, we make the following assumption\footnote{Actually, the convexity assumption on $\cR$ is only used from \cref{lem:stationarity_supremum} onwards.} on the risk measure $\cR$ and the regularization term $\varrho$.
\begin{assumption}\label{ass:onBothFunctionals}
Let $\cR\colon \Lone S \to \R$ be lower semicontinuous and convex and let $\varrho \colon \Linf{\Omega; [0,1]} \to \R$ be  weak-$\ast$ lower semicontinuous. 
 \end{assumption}

Under these assumptions, the relaxed problem \eqref{eq:min_problem_intro}, which we recall again
\begin{equation}%
\min_{(\theta, \bA) \in \cA} \cR[\cJ(\theta, \bA)] + \varrho(\theta)\tag{\ref{eq:min_problem_intro}},
\end{equation}
is a well defined problem. Let us now show that it possesses a solution.

\begin{theorem}\label{lemma:exist}
Suppose that  \cref{ass:on_gs_first,ass:onBothFunctionals} are satisfied. Then, the relaxed problem \eqref{eq:min_problem_intro} admits a solution $(\theta^*, \bA^*) \in \cA$.
\end{theorem}
\begin{proof}
Let us first show that $\cR \circ \cJ \colon \cA \to \R$ is sequentially lower semicontinuous. 
Consider a sequence $(\theta_n, \bA_n) \in \cA$ converging to $(\theta, \bA)$ with respect to the appropriate topologies, so that $\theta_n \weakstar \theta$ in $\Linf\Omega$ and $\bA_n \weakH \bA$. Let us set $\vu^n = G(\bA_n)$ and $\vu = G(\bA)$.  
By $H$-convergence  we have the weak convergence %
\begin{equation}\label{slun}
u_i^n(s,\cdot)\rightharpoonup u_i(s,\cdot) \quad { \rm in}\quad \Hzo\Omega\,,\quad {\rm a.s.}%
\end{equation}
Moreover, by the  Sobolev imbedding theorem  we have $\Hzo\Omega \subset \Lp q\Omega$ compactly for any $q\in[1,q_\ast)$ 
with $q_\ast$  as defined  above. Hence the convergence is actually strong in $\Lp q\Omega$.

Furthermore,  a.s. in $s\in S$, applying the \emph{a priori} estimate \eqref{eq:a_priori_u_i} for \eqref{ses}, we have
\begin{equation}\label{apriori}
\|u^n_i(s,\cdot)\|_{\Hzo\Omega}\leq \alpha^{-1}\| f_i(s,\cdot)\|_{\Hmo\Omega}\,,\;i=1,\ldots, m\,.
\end{equation}
Using the Sobolev embedding mentioned above, we further obtain a.s.
\[\|u^n_i(s,\cdot)\|_{\Lp q\Omega}\leq C\alpha^{-1}\| f_i(s,\cdot)\|_{\Hmo\Omega}\,,\;i=1,\ldots, m\,
\]
for any $q \in [1,q_\ast)$. %
 The right-hand side of this inequality (for any $i$) belongs to $\Lp pS$. Combined with the strong $\Lp q\Omega$ convergence of $u^n_i(s,\cdot)$ (for $q \in [1,q_\ast)$) %
 and  Lebesgue's dominated convergence theorem \cite[Proposition 1.2.5]{HNVW}, this implies  the strong convergence of 
$\vu^n$ to $\vu$ in $\Lp p{S;\Lp q{\Omega;\R^m}}$.   Consequently, there exists a function $U\in\Lp p{S;\Lp q{\Omega}}$ such that, for a subsequence that we shall relabel, $|\vu^n|\leq U$ almost surely-almost everywhere on $S\times\Omega$.

It follows that $\vu_n \to \vu$ in $\Lp {\min(p,q)}{S \times \Omega;\R^m}$. Hence, up to a further subsequence,  $\vu_n(s,\mx)\to \vu(s,\mx)$ for a.s.-a.e. $(s,\mx)\in S\times\Omega$, and therefore
\begin{equation}\label{ggamap}
g_\gamma (s,\mx,\vu_n(s,\mx))\to g_\gamma(s,\mx,\vu(s,\mx))\,, \;\gamma\in\{\alpha,\beta\}, \hbox{ a.s.-a.e. } (s,\mx)\in S\times\Omega\,,
\end{equation}
since $g_\alpha$ and $g_\beta$ are Carath\'eodory functions.  
  By the growth conditions (\ref{growth1}), we have
\begin{equation}\label{growthproof}
|g_\gamma (s,\mx,\vu_n(s,\mx))|\leq \varphi_\gamma(s,\mx) + \psi_\gamma(s,\mx){U(s,x)}^{q_1}\,.
\end{equation}
Since $U^{q_1}\in\Lp {\frac p{q_1}}{S;\Lp {\frac q{q_1}}\Omega}$, an application of H\"older's inequality to the second term on the right-hand side (see \cite[Lemma IV.1.4]{GGZ}, or the analogus result for the mixed-norm Lebesgue spaces \cite{BP})
yields $\psi_\gamma U^{q_1}\in\Lone{S\times\Omega}$ if
\[
\frac 1{q_2}+\frac {q_1}p=1\,,\; \frac 1{q_3}+\frac {q_1}q=1
\]
which hold by assumption.%

Since the right-hand side of \eqref{growthproof} belongs to $\Lp 1{S\times\Omega}$, the pointwise convergence in \eqref{ggamap}, together with the dominated convergence theorem, implies that 
$g_\gamma (\cdot,\cdot,\vu_n)$
converges to $g_\gamma (\cdot,\cdot,\vu)$ strongly in $\Lone{S\times\Omega}$, for $\gamma\in\{\alpha,\beta\}$.

Finally, since $\theta_n\weakstar\theta$ in  $\Linf\Omega$ we conclude $\cJ(\theta_n,\bA_n)\to \cJ(\theta,\bA)$ in $\Lone S$. Indeed, we observe that
\begin{align*}
&\int_S \left|\cJ(\theta_n, \bA_n)- \cJ(\theta, \bA)\right|\,\mathrm{d}\mathbb{P}(s)\\
&=     \int_S \left|\int_\Omega \Bigl(\theta_n\, g_\alpha(\vu_n) + (1-\theta_n)\, g_\beta(\vu_n)\Bigr) - \Bigl(\theta g_\alpha(\vu) + (1-\theta)g_\beta(\vu)\Bigr)\,\mathrm{d}\mx\,\right|\mathrm{d}\mathbb{P}(s)\\
&=     \int_S \Big|\int_\Omega \Bigl(\theta_n\, (g_\alpha(\vu_n)-g_\alpha(\vu)) + (1-\theta_n)\, (g_\beta(\vu_n)-g_\beta(\vu))\Bigr)\,\mathrm{d}\mx \,\\
&\qquad+ \int_\Omega \Bigl((\theta_n-\theta)\, g_\alpha(\vu) + (\theta-\theta_n)\, g_\beta(\vu)\Bigr)\,\mathrm{d}\mx\,\Big|\,\mathrm{d}\mathbb{P}(s)\\
&\leq     \int_S \int_\Omega \left|\theta_n\, (g_\alpha(\vu_n)-g_\alpha(\vu))\right|\,\mathrm{d}\mx \,\mathrm{d}\mathbb{P}(s) + \int_S \int_\Omega \left|(\theta_n-\theta)\, g_\alpha(\vu)\right|\,\mathrm{d}\mx \,\mathrm{d}\mathbb{P}(s) \\
&\qquad + \int_S\int_\Omega \left| (1-\theta_n)\, (g_\beta(\vu_n)-g_\beta(\vu))\right|\,\mathrm{d}\mx \,\mathrm{d}\mathbb{P}(s)\\
&\qquad+ \int_S\int_\Omega \left|(\theta-\theta_n)\, g_\beta(\vu)\,\right|\,\mathrm{d}\mx \,\mathrm{d}\mathbb{P}(s)\\
&= \mathrm{I} + \mathrm{II} + \mathrm{III} + \mathrm{IV}.
\end{align*}
Now, we have
\begin{align*}
    \mathrm{I} &=\int_S\int_\Omega  \left|\theta_n\, (g_\alpha(\vu_n)-g_\alpha(\vu))\right| \,\mathrm{d}\mx \,\mathrm{d}\mathbb{P}(s) \\
    &\leq \lVert \theta_n \rVert_{L^\infty(\Omega)}\lVert g_\alpha(\vu_n)-g_\alpha(\vu)\rVert_{L^1(S\times\Omega)} \to 0
\end{align*}
since $(\theta_n)$ is uniformly bounded in $L^\infty$ due to weak-$\ast$ convergence, and the second factor tends to zero by the strong $\Lone {S\times\Omega}$ convergence established above. Next, consider the term $\mathrm{II}$ above.
Using the weak-$\ast$ convergence, we find 
\[h_n(s) := \int_\Omega |(\theta_n-\theta)\, g_\alpha(\vu)(s)| \,\mathrm{d}\mx\to 0 \quad\text{pointwise a.s. in $s$.}\]
Moreover, we have the bound 
\begin{align*}
h_n(s)  \leq C    \int_\Omega |g_\alpha(\vu)(s)|\,\mathrm{d}\mx \leq C\int_\Omega \varphi_\gamma(s) + \psi_\gamma(s){|\vu(s)|}^{q_1} \,\mathrm{d}\mx
\end{align*}
by the growth condition \eqref{growth1}. The right-hand side belongs to $L^1(S)$ by assumption, as explained previously. Therefore, the dominated convergence theorem implies that $h_n \to 0$ in $L^1(S)$, i.e.,  $\int_S h_n(s) \to 0$, which shows that $\mathrm{II}$ converges to 0. The terms $\mathrm{III}$ and $\mathrm{IV}$ can be treated similarly. This establishes that $J(\theta_n,\bA_n)\to J(\theta,\bA)$ in $\Lone S$.

As mentioned before, the  topology introduced on $\cA$ is metrizable, which enables us to conclude that $\cJ\colon \cA\to\Lone S$ is continuous. 
It follows that the sum $F:= \cR \circ \cJ + \varrho$ is also lower semicontinuous with respect to weak$-*$ topology for $\theta$ and $H$-topology for $\bA$.
Now the result follows by the Weierstrass theorem, see \cite[\S7, Theorem 7.3.1]{CFA}. %
\end{proof} 

\section{Necessary conditions of optimality}\label{sec:optimality}
In this section we will derive first-order stationarity conditions for \eqref{eq:min_problem_intro}. We will denote by $(\ths,\As)$  an optimal relaxed solution of \eqref{eq:min_problem_intro}, the existence of which is assured by \cref{lemma:exist}. We need additional regularity on the functions $g_\alpha$ and $g_\beta$ in the form of the next assumption.
\begin{assumption}\label{ass:on_g_growth_second}
The functions $\frac{\partial g_\alpha}{\partial u_i}$ and $\frac{\partial g_\beta}{\partial u_i}$ are supposed to be Carath\'eodory functions (measurable in $(s,\mx)$ and continuous in $\vu$) satisfying the  growth conditions
\begin{equation}\label{growder}
\left|\frac{\partial g_\gamma}{\partial u_i}(s,\mx,\vu)\right|\leq \widetilde\varphi_\gamma(s,\mx) 
+ \widetilde\psi_\gamma(s,\mx){|\vu|}^{q_1-1} \qquad{\rm for}\; 
\gamma=\alpha,\beta; i=1,\ldots, m, 
\end{equation}
where $\widetilde\varphi_\gamma\in\pL{p'}{S;\pL{q_\ast'}{\Omega}}$, 
$\widetilde\psi_\gamma\in\pL{q_4}{S;\pL{{q_5}}{\Omega}}$ with $1\leq q_1\leq\min\{p,q_\ast\}$, $q_4=\displaystyle\frac{p}{p-q_1}$,  and
$q_5=\displaystyle\frac{q_*}{q_*-q_1}$.
\end{assumption}
In \cite{MT} (and similarly \cite[Theorem 3.2.13]{A} for multiple state problems), the G\^{a}teaux differential of the  functional $\cJ(\cdot,\cdot)(s)\colon \cA\to\R$ is computed for a.s. $s\in S$. Using this and arguing like in the proof of \cref{lemma:exist}, we obtain the G\^{a}teaux differentiability of $\cJ\colon \cA\to\Lone S$ (see also \cite{Vran}):
\begin{equation}\label{gat}
\begin{aligned}
\cJ'(\ths,\As)(\dth,\dA)(s)&=\int_\Omega \dth\bigl[g_\alpha(s,\mx,\us(s,\mx))-g_\beta(s,\mx,\us(s,\mx))\bigr]\,\mathrm{d}\mx\\
        &\quad-  \int_\Omega \sum_{i=1}^m\dA\nabla u_i^\ast(s,\mx)\cdot\nabla p_i^\ast(s,\mx)\,\mathrm{d}\mx,
\end{aligned}
\end{equation}
where $\ps=(p_1^\ast,\ldots, p_m^\ast)$ denotes the adjoint state, defined as the unique solution\footnote{Existence and uniqueness for \eqref{adjoint} can be proved in a similar fashion as \cref{lemma:exist}; we omit the details.} of  
\begin{equation}\label{adjoint}
\left\{
\begin{array}{l}
-\div(\As\nabla p_i)=\ths\frac{\partial g_\alpha}{\partial u_i}(\cdot,\cdot,\us)+(1-\ths)\frac{\partial g_\beta}{\partial u_i}(\cdot,\cdot,\us)\\
p_i\in \Lp {p'}{S;\Hzo\Omega}\\
\end{array}
\right.\;i=1,\ldots, m.
\end{equation}

\smallskip

In fact, we will prove that $\cJ$ is differentiable in the Fr\'echet sense (see also \cite[Proposition 7]{MT}). This will be used later when we come to apply the chain rule. For this purpose, we need the following technical lemma.
\begin{lemma}\label{strong}
    If $\bA_n \to \bA$ in $\Linf{\Omega;\cM_{\alpha, \beta}}$ and $f_n\to f \in \Hmo\Omega$, the sequence $(u_n)$ of solutions to 
    \[
    \left\{
\begin{array}{l}
-\div(\bA_n\nabla u_n)=f_n\\
u_n\in {\rm H}^1_0(\Omega)\,\\
\end{array}
\right.
    \]
    converges strongly to $u$ in $\Hzo\Omega$, where $u$ is the solution of \eqref{sdif}.
\end{lemma}
\begin{proof}
    The boundedness of $\Omega$ implies that $\bA_n \to \bA$ in $\Lone{\Omega;{\rm M}_d(\R)}$, and hence, by \cite[Lemma 1.2.22]{A}, $\bA_n \to \bA$ in the sense of $H$-convergence. Hence $u_n \rightharpoonup u$ in $\Hzo\Omega$. To establish strong convergence,  first write
    \[-\div (\bA \nabla u_n) = f_n-\div ((\bA-\bA_n)\nabla u_n) =: g_n.\]
    Since $\nabla u_n \rightharpoonup \nabla u$ in $L^2(\Omega)$, the sequence $(\nabla u_n)$  is bounded in $L^2(\Omega)$, and we deduce that $(\bA-\bA_n)\nabla u_n \to 0$ in $L^2(\Omega; \mathbb{R}^d)$. This implies  $\nabla \cdot (\bA-\bA_n)\nabla u_n \to 0$ in $\Hmo\Omega$, i.e., $g_n \to f$ in $\Hmo\Omega$. Finally,  since the mapping $u\mapsto -\div (\bA \nabla u)$ is an isomorphism between $\Hzo\Omega$ and $\Hmo\Omega$, we conclude that $u_n \to u$ strongly in $\Hzo\Omega$.
\end{proof}
\begin{lemma}\label{lem:J_is_C1}
Let \cref{ass:on_gs_first,ass:onBothFunctionals,ass:on_g_growth_second} hold. Then, the mapping $\cJ\colon \cA\to\Lone S$ is continuously differentiable. 
\end{lemma}
\begin{proof}
Let $\bA_n\to \bA$ strongly in $\Linf{\Omega;\cM_{\alpha, \beta}}$ and $\theta_n\to\theta$ strongly in $\Linf{\Omega}$, and denote the corresponding states $\vu^n = G(\bA_n)$ and $\vu = G(\bA)$.
Our goal is to show that $J'(\theta_n,\bA_n)$ converges to $J'(\theta,\bA)$ strongly in the operator norm. To this end, observe that
\begin{align*}
&\lVert (\cJ'(\theta_n,\bA_n)-\cJ'(\ths, \As))(\dth,\dA)\rVert_{L^1(S)}\\
&\quad= \int_S \Big|\int_\Omega \dth\bigl[g_\alpha(\vu^n)-g_\beta(\vu^n) - (g_\alpha(\us)-g_\beta(\us))\bigr]\,\mathrm{d}\mx\\
&\quad\quad-  \int_\Omega \sum_{i=1}^m\dA\nabla u_i^n\cdot\nabla p_i^n - \sum_{i=1}^m\dA\nabla u_i^\ast\cdot\nabla p_i^\ast\,\mathrm{d}\mx\Big|\,\mathrm{d}\mathbb{P}(s)\\
&\quad\leq \int_S \int_\Omega |\dth\bigl[g_\alpha(\vu^n)-g_\beta(\vu^n) - (g_\alpha(\us)-g_\beta(\us))\bigr]|\,\mathrm{d}\mx \,\mathrm{d}\mathbb{P}(s)\\
&\quad\quad+  \int_S \int_\Omega \sum_{i=1}^m|\dA\nabla u_i^n\cdot\nabla p_i^n - \dA\nabla u_i^\ast\cdot\nabla p_i^\ast|\,\mathrm{d}\mx\,\mathrm{d}\mathbb{P}(s)\\
&\quad\leq \lVert \dth\rVert_{L^\infty(\Omega)}\int_S \int_\Omega |g_\alpha(\vu^n)-g_\beta(\vu^n) - (g_\alpha(\us)-g_\beta(\us))|\,\mathrm{d}\mx\,\mathrm{d}\mathbb{P}(s)\\
&\quad\quad+  \lVert \dA\rVert_{\Linf{\Omega; \Sym}}\int_S \int_\Omega \sum_{i=1}^m| \nabla u_i^n\cdot\nabla p_i^n -  \nabla u_i^\ast\cdot\nabla p_i^\ast|\,\mathrm{d}\mx\,\mathrm{d}\mathbb{P}(s).
\end{align*} 
The first integral on the right-hand side can be shown to vanish by following the argument in the proof of  \cref{lemma:exist}, which yields the strong convergence $g_\alpha(\cdot,\cdot,\vu_n)-g_\beta(\cdot,\cdot,\vu_n)$ to
$g_\alpha(\cdot,\cdot,\vu)-g_\beta(\cdot,\cdot,\vu)$  in $\Lone{S\times\Omega}$.%

For the second integral, \cref{strong} shows that
$u_i^n(s,\cdot)\to u_i(s,\cdot)$ in $\Hzo\Omega$ for each $i$ and almost surely in $s\in S$. As in the proof of \cref{lemma:exist}, using 
the \emph{a priori} bound \eqref{eq:a_priori_u_i} and the dominated convergence theorem, we obtain the strong convergence $\vu_n\to\vu$ in $\Lp p{\Omega;\Hzo{\Omega;\R^m}}$, i.e.~$\nabla u^n_i\to\nabla u_i$ in $\Lp p{\Omega;\Lp 2{\Omega;\R^n}}$. To obtain an analogous conclusion for the adjoint functions $p_n$, we first show the strong convergence 
\begin{equation}\label{derpc}
\frac{\partial g_\gamma}{\partial u_i} (\cdot,\cdot,\vu_n)\to \frac{\partial g_\gamma}{\partial u_i}(\cdot,\cdot,\vu)\, \;\hbox{in } \Lp {p'}{S;\Lp {q_\ast'}\Omega}\;\hbox{for}\;\gamma\in\{\alpha,\beta\}\,.
\end{equation}
This can be proved in a manner analogous to the proof of \cref{lemma:exist}. The only difference occurs in the step following \eqref{apriori}, where we may take $q_\ast$ instead of $q\in[1,q_\ast)$, %
since the compactness of the embedding $\Hzo\Omega\hookrightarrow \Lp {q}\Omega$ is not required. As a result, 
$\vu_n\to\vu$ strongly in $\Lp p{S;\Lp {q_\ast}{\Omega;\R^m}}$, and passing to a subsequence, $|\vu_n|\leq V$ almost surely-almost everywhere on $S\times\Omega$ for some $V\in\Lp p{S;\Lp {q_\ast}{\Omega}}$.
Now, along a further subsequence, we have pointwise a.s.-a.e. convergence of the convergence that appears in \eqref{derpc}, and by the growth conditions (\ref{growder}), 
\[
\left|\frac{\partial g_\gamma}{\partial u_i}(s,\mx,\vu_n(s,\mx))\right|\leq \widetilde\varphi_\gamma(s,\mx) + \widetilde\psi_\gamma(s,\mx){V(s,x)}^{q_1-1}\,.
\]
We have $\displaystyle V^{q_1-1}\in\Lp {\frac p{q_1-1}}{S;\Lp {\frac {q_\ast}{q_1-1}}\Omega}$, so  by  H\"older's inequality %
 $\widetilde\psi_\gamma V^{q_1-1}\in\Lp {p'}{S;\Lp {q_\ast'}\Omega}$ provided that
\[
\frac 1{q_4}+\frac {q_1-1}p=\frac 1{p'}=\frac{p-1}p\,,\;\, \frac 1{q_5}+\frac {q_1-1}{q_*}=\frac 1{q_*'}=\frac{q_*-1}{q_*}
\]
which hold by assumption \eqref{growder}. Applying the dominated convergence theorem then yields \eqref{derpc}.

Then, the right-hand side in the adjoint equation %
\[
\left\{
\begin{array}{l}
-\div(\bA_n\nabla p^n_i)=\theta_n\frac{\partial g_\alpha}{\partial u_i}(\cdot,\cdot,\vu_n)+(1-\theta_n)\frac{\partial g_\beta}{\partial u_i}(\cdot,\cdot,\vu_n)\\
p^n_i\in \Lp {p'}{S;\Hzo\Omega}\\
\end{array}
\right.\qquad i=1,\ldots, m,
\]
belongs to  $\Lp {p'}{S;\Lp {q_\ast'}\Omega}$, which is continuously embedded in $\Lp {p'}{S;\Hmo\Omega}$, yielding the unique solution $p^n_i\in\Lp {p'}{S;\Hzo\Omega}$. Moreover, from \eqref{derpc} and \cref{strong}, we have convergence $\vp_n(s,\cdot)\to\vp(s,\cdot)$ in $\Hzo\Omega$ almost surely in $s$. As shown for the sequence $(\vu_n)$ in the proof of \cref{lemma:exist}, this implies the strong convergence $\vp_n\to\vp$ in $\Lp {p'}{\Omega;\Hzo{\Omega;\R^m}}$.

Finally, we conclude that $\nabla u^n_i\cdot\nabla p^n_i\to\nabla u_i\cdot\nabla p_i$ strongly in $\Lp {1}{S\times\Omega}$, which completes the proof.
\end{proof}
To write down the optimality conditions for \eqref{eq:min_problem_intro}, we need some further assumptions on $\varrho$. One of those requires the notion of \emph{regularity} in the sense of Clarke \cite[Definition 2.3.4, p.~39]{Clarke}. Indeed, the function $\rho$ is said to be \emph{regular} at $\theta$ if the (usual one-sided) directional derivative at $\theta$ exists for all directions $v$ and it equals the \emph{generalized directional derivative} $\rho^{\circ}(\theta; v) = \limsup_{\eta \to \theta, t \downarrow 0}(\rho(\eta + tv)-\rho(\eta))\slash t.$

We make the following assumption.
\begin{assumption}\label{ass:new_on_varrho}
\begin{enumerate}[label=(\roman*)]\itemsep=0em
    \item\label{item:Hdiff} 
Let  $\varrho\colon \Linf{\Omega;[0,1]} \to \R$ be Hadamard directionally differentiable at $\ths$. 
\item\label{item:ClarkeReg} Let $\varrho$  be locally Lipschitz and regular at $\theta^*$.
\end{enumerate}
\end{assumption}
Let us try to write down a first optimality condition.  We argue in a similar fashion to \cite{AGHS}. We start by recalling the definition of the tangent (Bouligand) cone of $\cA$ at the point $(\theta, \bA)$:
\begin{align*}
&T_{\cA}(\theta, \bA)
:=  \{ (h, d)  \in \Linf{\Omega; [0,1]}\times \Linf{\Omega; \Sym} \mid \exists s_k \searrow 0, \;\\
&\quad  \exists (h_k, d_k) \to (h,d), %
(\theta, \bA) + s_k(h_k, d_k) \in \cA \hbox{ for each } k\}.
\end{align*}
\begin{lemma}\label{lem:stationarity_supremum}
Let \cref{ass:on_gs_first,ass:onBothFunctionals,ass:on_g_growth_second} and \cref{ass:new_on_varrho} \ref{item:Hdiff} hold. If $(\ths,\As)$ is an optimal relaxed solution of \eqref{eq:min_problem_intro}, then, for every $(\dth, \dA) \in T_{\cA}(\theta^*, \bA^*)$ we have
\[\sup_{\nu \in \partial\cR(\cJ(\ths,\As))} \mathbb{E}[\nu \cJ'((\ths,\As))(\dth,\dA) ]   + \varrho'(\theta^*)(\dth) \geq 0.\]
\end{lemma}
\begin{proof}
    Firstly, since $\cR$ is finite, lower semicontinuous and convex, $\cR$ is continuous \cite[Proposition 2.111]{Bonnans2013}, %
    and by \cite[Proposition 2.126 (v)]{Bonnans2013}, we obtain that in fact $\cR$ is Hadamard directionally differentiable with
\[\cR'[z](h) = \sup_{\nu \in \partial\cR(z)} \mathbb{E}[\nu h ].\]
Now we can use the chain rule \cite[Proposition 2.47]{Bonnans2013}: since $\cJ\colon \cA \to \Lone\Omega$ is $C^1$ by \cref{lem:J_is_C1} and hence also Hadamard directionally differentiable, and $\cR$ is Hadamard directionally differentiable, we find that the composition $\cR \circ \cJ$ is Hadamard directionally differentiable too with 
\[(\cR \circ \cJ)'(\ths,\As)(\dth,\dA)= \cR'[\cJ(\ths,\As)]\cJ'((\ths,\As))(\dth,\dA), %
\]
for every $(\dth, \dA) \in T_{\cA}(\theta^*, \bA^*)$.
The set $\cA \subset \Linf{\Omega; [0,1]}\times \Linf{\Omega; \Sym}$ is in general non-convex, hence the closure of the radial cone may not be the tangent cone, so the usual density argument to obtain the claim of this result has to be adapted. 

Now, if $(\dth, \dA) \in T_{\cA}(\theta^*, \bA^*)$, by definition, we have the existence of $\{(h_k, d_k)\}$ such that $(h_k, d_k) \to (\dth, \dA)$ and a null sequence $\{s_k\}$ with
\[\cR[\cJ(\theta^* + s_kh_k, \bA^*+s_kd_k)] + \varrho(\theta^* + s_kh_k) -\cR[\cJ(\theta^*, \bA^*)] - \varrho(\theta^*) \geq 0,\]
since $(\theta^*, \bA^*)$ is optimal. Dividing the above expression by $s_k$ and taking the limit $k \to \infty$ and using the fact that $\cR \circ \cJ$ and $\varrho$ are Hadamard directionally  differentiable, we get
\[\cR'[\cJ(\ths,\As)](\cJ'((\ths,\As))(\dth,\dA)) + \varrho'(\theta^*)(\dth) \geq 0.\]
Using the characterisation of $\cR'$ in terms of the expectation stated above, we obtain the result. %
\end{proof}
As done in \cite{Corrigendum}, we wish to reformulate the above condition since the presence of the supremum  renders it inconvenient and makes it less tractable.  By means of subdifferential calculus, this reformulation is possible as we see in the next result. This result enables the adaptation of the optimality criteria method which is well established in the deterministic setting to the risk-averse setting as well. Below, $\partial \varrho$ refers to the generalized gradient in the sense of Clarke, see \cite[p.~27]{Clarke}.

\begin{theorem}\label{prop:explicit_stat_condition}
Let $(\ths,\As)$ be an optimal relaxed solution of \eqref{eq:min_problem_intro}. Let \cref{ass:on_gs_first,ass:onBothFunctionals,ass:on_g_growth_second,ass:new_on_varrho}  hold. 
Then, there exists $\pi^* \in  \partial \cR(\cJ(\theta^*, \bA^*))$ such that  for every $(\dth, \dA) \in T_{\cA}(\theta^*, \bA^*)$,
\begin{equation}\label{eq:neccon}
\begin{aligned}
 0\leq \varrho'(\theta^*)(\dth) &+\int_S \int_\Omega \pi^*\dth\bigl[g_\alpha(s,\mx,\us)-g_\beta(s,\mx,\us)\bigr]\,\mathrm{d}\mx\,\mathrm{d}\mathbb{P}(s)\\
       & -  \int_S \int_\Omega \pi^*\sum_{i=1}^m\dA\nabla u_i^\ast\cdot\nabla p_i^\ast\,\mathrm{d}\mx \,\mathrm{d}\mathbb{P}(s).     
\end{aligned}
\end{equation}
\end{theorem}
\begin{proof}
The proof is similar to that of \cite[Lemma 3.11]{AGHS} with some adaptations to deal with the non-convexity.  To begin, observe that $\cR$ is %
locally Lipschitz near every point of $\pL 1\Omega$ and that $\mathcal{J} \colon \cA \rightarrow \pL 1\Omega$, being $C^1$, is also strictly differentiable  (in the sense of Clarke) \cite[Corollary, p.~32]{Clarke}. By \cite[Proposition 2.2.1]{Clarke}, $\cJ$ is Lipschitz near $(\theta^*, \bA^*)$. Therefore, the composition $\cR\circ \cJ$ is locally Lipschitz near $(\theta^*, \bA^*)$.

Now, $\varrho$ is assumed locally Lipschitz on $\Linf{\Omega;[0,1]}$. In combination with the above, we thus have have that the sum $\cR\circ \cJ + \varrho$ is also locally Lipschitz. Invoking the corollary on page 52 of \cite{Clarke}, we get  
\[0 \in \partial(\cR \circ \cJ + \varrho)(\theta^*, \bA^*) + N_{\cA}(\theta^*, \bA^*)\]
where $N_{\cA}(u) = T_{\cA}(u)^\circ = \{ \eta \in (\Linf{\Omega; [0,1]}\times \Linf{\Omega; \Sym})^*: \langle \eta, v \rangle \leq 0\,, \; v \in T_{\cA}(u)\}$ is the normal cone of $\cA$ at the point $u$. Applying \cite[Proposition 2.3.3]{Clarke}, the first term on the right-hand side above can be expressed as
\[ \partial(\cR \circ \cJ + \varrho)(\theta^*, \bA^*)  \subset  \partial(\cR \circ \cJ )(\theta^*, \bA^*)+ \partial \varrho(\theta^*).\]
Since $\cR$ is locally Lipschitz on $\pL 1\Omega$, we are able to utilize the subdifferential chain rule \cite[Theorem 2.3.10 and Remark 2.3.11]{Clarke} to find
\[\partial (\cR\circ \cJ )(\theta^*, \bA^*)= [\cJ'(\theta^*, \bA^*)]^*\partial \cR(\cJ (\theta^*, \bA^*))\]
(equality holds since $\cR$ is regular \cite[Proposition 2.3.6 (b)]{Clarke} by virtue of its convexity). Gathering everything together, we have
\[0 \in [\cJ'(\theta^*, \bA^*)]^*\partial \cR(\cJ(\theta^*, \bA^*)) + \partial \varrho(\theta^*) + N_{\cA}(\theta^*, \bA^*).\]
Now, as in the proof of \cite[Lemma 3.11]{AGHS}, we follow \cite{Corrigendum}. From the above inclusion, we obtain the existence of $\pi^* \in \partial \cR(\cJ(\theta^*, \bA^*))$ and $\eta \in \partial\varrho(\theta^*)$ such that
\[-[\cJ'(\theta^*, \bA^*)]^*\pi^*  - \eta  \in N_{\cA}(\theta^*, \bA^*).\] This reads
\[\langle -[\cJ'(\theta^*, \bA^*)]^*\pi^*  - \eta, v \rangle \leq 0\,, \;  %
v \in T_{\cA}(\theta^*, \bA^*).\]
Now, since $\varrho$ is regular, by definition of the generalized gradient, we have that%
\begin{align*}
\partial \varrho(\theta^*) = \{ \xi \in (\Linf{\Omega; [0,1]})^* : \varrho'(\theta^*)(v) \geq \langle \xi, v \rangle, v \in \Linf{\Omega; [0,1]}\}.    
\end{align*}
Thus the above can be written as
\[\langle -[\cJ'(\theta^*, \bA^*)]^*\pi^*  - \varrho'(\theta^*), v \rangle \leq 0\,, \; %
v \in T_{\cA}(\theta^*, \bA^*).\]
Then
\[\mathbb{E}[\cJ'(\theta^*, \bA^*)v \pi^*]     +  \varrho'(\theta^*)(\dth)\geq 0\,, \; %
v =(\dth, \dA) \in T_{\cA}(\theta^*, \bA^*)\]
follows upon realizing $\langle \cJ'(\theta^*, \bA^*)^*\pi^*, v \rangle =\langle \pi^*, \cJ'(\theta^*, \bA^*)v \rangle $ as the $L^1(S)$-$L^1(S)^*$ pairing.
The claim of the theorem is a direct consequence of substituting the derivative formula  \eqref{gat}. 
\end{proof}

To further analyze the optimality condition of \cref{prop:explicit_stat_condition}, let us specialize to the case where the regularization term is of the form
\begin{equation}
\varrho(\theta)=\int_\Omega h(\theta)\,\mathrm{d}\mx\label{eq:form_of_rho},  
\end{equation}
where $h\colon[0,1]\to\R$ is a continuously differentiable function. In classical applications, a choice such as %
$\varrho(\theta)=\lambda\int_\Omega \theta\,\mathrm{d}\mx$ penalizes the amount  of the first phase. For greater generality, we consider instead the form stated above.

We can analyze the result of \cref{prop:explicit_stat_condition} using the same approach as in the deterministic setting \cite{MT,A}. 
The dependence on the variable 
$s$ in the optimality conditions can be addressed  in a manner reminiscent of the treatment of the time variable  in evolutionary problems \cite{MV_hyp}. We begin by fixing $\theta$ and considering variations only in $\bA$. Since $K(\vartheta)$ is convex for every $\vartheta\in[0,1]$, then for every $\bA\in K(\ths)$ we have
$(0, \bA-\As) \in T_{\cA}(\ths, \As)$.

Therefore, we obtain
\begin{align*}   
-  \int_S \int_\Omega \pi^*\sum_{i=1}^m(\bA-\bA^*)\nabla u_i^\ast\cdot\nabla p_i^\ast\,\mathrm{d}\mx\,\mathrm{d}\mathbb{P}(s)    \geq 0.
\end{align*}
By interchanging the order of integration, this may be rewritten as
\begin{align*}   
-  \int_\Omega \int_S \pi^*\sum_{i=1}^m(\bA-\bA^*)\nabla u_i^\ast\cdot\nabla p_i^\ast \,\mathrm{d}\mathbb{P}(s)\,\mathrm{d}\mx   \geq 0.
\end{align*}
We define  
\begin{equation}\label{eq:Ma}
\Ms:= {\rm Sym}\int_S \pi^* \sum_{i=1}^m \nabla u_i^\ast\otimes\nabla p_i^\ast\,\mathrm{d}\mathbb{P}(s),
\end{equation}
where ${\rm Sym}$ denotes the symmetric part of a matrix, and $\otimes$ the tensor product of two vectors in $\R^d$. 
With this notation, the above condition can be rewritten as
\begin{align*}   
\int_\Omega \bA^* : \Ms \,\mathrm{d}\mx     \geq \int_\Omega \bA : \Ms\,\mathrm{d}\mx  ,
\end{align*}
where $:$ denotes the Euclidean inner product on $\mathrm{M}_d(\R)$.
Since $\bA\in K(\ths)$ is arbitrary, we may choose $\bA$ to equal $\As$ a.e.  except in a small neighborhood of an arbitrary point. This yields the following pointwise estimate: almost everywhere on $\Omega$,  $ \bA^* : \Ms \geq \bA : \Ms$ for any $\bA\in K(\ths)$. To summarize, the necessary optimality condition can be expressed as follows: almost everywhere on $\Omega$, the optimal $\As$ satisfies
\[
 \bA^* : \Ms={F}(\ths,\Ms)\,,
\]
where ${F}\colon[0,1]\times\mathrm{M}_d(\R)\to\R$ is given by ${F}(\vartheta,\bM)=\max_{\bA\in K(\vartheta)} \bA:\bM$. %

We can now proceed as in \cite[Theorem 3.2.14]{A} and consider variations in $\theta$. Since the function ${F}$ is continuously differentiable with respect to $\theta$, we can consider a smooth path $\theta(t)\in\Linf{\Omega; [0,1]}$ passing through $\ths$ with derivative $\dth$, and choose $\bA(t)\in K(\theta(t))$ such that $\bA(t): \Ms={F}(\theta(t),\Ms)$, almost everywhere on $\Omega$.  
Consequently, the derivative of $\bA$ with respect to $t$, evaluated at $t=0$ where $\theta(0)=\ths$, satisfies
\[
\dA : \Ms=\frac{\partial {F}}{\partial \theta}(\ths,\Ms)\dth.
\]
The partial derivative of ${F}$ with respect to $\theta$ can be expressed explicitly in the two- and three-dimensional case \cite{AVmpe, V}. Finally, we come to the following result.

\begin{theorem}
Let $\varrho$ be of the form \eqref{eq:form_of_rho} and let the assumptions of \cref{prop:explicit_stat_condition} hold. If $(\ths,\As)$ is an optimal relaxed solution of \eqref{eq:min_problem_intro} and $\us$ and $\ps$ the corresponding state and adjoint state functions respectively, then, there exists $\pi^* \in  \partial \cR(\cJ(\theta^*, \bA^*))$ such that 
for $Q\colon\Omega\to\R$ defined by
\[
Q=\int_S \pi^* \bigl[g_\alpha(s,\cdot,\us)-g_\beta(s,\cdot,\us)\bigr]\,\mathrm{d}\mathbb{P}(s)-\frac{\partial {F}}{\partial \theta}(\ths,\Ms)+h'(\theta),
\]
the following  necessary optimality condition hold: for almost every $\mx\in\Omega$
\begin{align*}   
Q(\mx)&> 0\; \Longrightarrow\; \ths(\mx)=0\\
Q(\mx)&< 0\; \Longrightarrow\; \ths(\mx)=1.%
\end{align*}
\end{theorem}

\section{Numerics via the optimality criteria method}\label{sec:numerics}

In the section, we will study a particular instance of \eqref{eq:min_problem_intro} and solve it numerically. 
We will take advantage of the first-order conditions we derived in the previous section and use the optimal criteria method (see e.g. \cite[Chapter 1]{BS03} for more details) in order to compute a solution of this problem. The method in our general problem setting is given in \cref{alg:OC}.

\begin{algorithm}[H]
\caption{Optimality Criteria Method}
\label{alg:OC}
\begin{algorithmic}
\STATE \textbf{Initialization.}
Choose an admissible initial design $(\theta^0,\bA^0)$.

\STATE \textbf{For} $k=0,1,2,\dots$, repeat the following steps:

\STATE\hspace{0.7cm} \textbf{Step 1 (State equation).}
Compute the state
\[
\vu^k=(u_1^k,\ldots,u_m^k)=G(\bA^k).
\]

\STATE\hspace{0.7cm} \textbf{Step 2 (Adjoint equation).}
Compute the adjoint variable
\[
\vp^k=(p_1^k,\ldots,p_m^k),
\]
\hspace{0.8cm}solving \eqref{adjoint} with $\bA^k$, $\theta^k$, and $\vu^k$ in place of
$\ths$, $\As$, and $\vu^*$.

\STATE\hspace{0.7cm} \textbf{Step 3 (Sensitivity matrix).}
Compute
\[
\bM^k
=
{\rm Sym}\int_S \pi^k \sum_{i=1}^m \nabla u_i^k \otimes \nabla p_i^k \,  \mathrm{d}\mathbb{P}(s),
\]
\hspace{0.8cm}where $\pi^k$ is a subgradient of $\cR$ at $\cJ(\theta^k,\bA^k)$.

\STATE\hspace{0.7cm} \textbf{Step 4 (Update of $\theta$).}
For each $\mx\in\Omega$, define $\theta^{k+1}(\mx)$ as a zero of
\begin{align}
\nonumber \theta \mapsto 
\int_S \pi^k \bigl[g_\alpha(s,\mx,\vu^k)-g_\beta(s,\mx,\vu^k)\bigr]\,\mathrm{d}\mathbb{P}(s) &-\frac{\partial {F}}{\partial \theta}(\theta,\bM^k(\mx))\\
&+h'(\theta).
\label{eq:qk_alg}
\end{align}
\hspace{0.8cm}If a zero does not exist, set $\theta^{k+1}(\mx)=0$ when the function is positive,

\hspace{0.8cm}and $\theta^{k+1}(\mx)=1$ when it is negative.

\STATE\hspace{0.7cm} \textbf{Step 5 (Update of $\bA$).}
For each $\mx\in\Omega$, choose
\[
\bA^{k+1}(\mx)\in
\arg\max_{\bA} {F}\bigl(\theta^{k+1}(\mx),\bM^k(\mx)\bigr).
\]

\end{algorithmic}
\end{algorithm}
We still need to specify for example how to numerically approximate the integrals over the probability space $S$ that appear in the algorithm; this will be addressed in the next section for our specific example.

\subsection{An example involving the conditional value-at-risk}

Let us now specify the exact problem we will simulate. For the risk measure $\cR$, we adopt the conditional (average) value-at-risk CVaR. To define this, let us recall some definitions from \cite[\S 6.2.3--6.2.4]{SDR14} for the convenience of the reader. Firstly, the \emph{value-at-risk (VaR)} at level $\gamma\in (0,1)$  %
for a random variable representing losses, also known as the $\gamma$-quantile, is defined as
\[
\varg (X)=\inf \{t:\mathbb{P}(X> t)\leq \gamma\}=\inf \{t:\mathbb{P}(X\leq t)\geq 1-\gamma\}\,.
\]
Throughout this text, we assume that the random variable $X$ has no probability atoms, i.e. its distribution is absolutely continuous with respect to the Lebesgue measure.
The \emph{conditional value-at-risk (CVaR)} at level $\gamma$ is the expected loss given that the loss exceeds the value-at-risk $\varg$ at level $\gamma$:
\begin{equation}
    \cvarg(X)= \mathbb{E}[X|X\geq\varg(X)].\label{eq:defnCVAR}
\end{equation}
Intuitively, $\cvarg$ can be understood as the expected loss in the worst $\gamma$-fraction or tail of outcomes.

To keep the notation simple, we consider the case of a single state problem (i.e., $m=1$)
\begin{equation}\label{sdifs}
\left\{
\begin{array}{l}
-\div(\bA\nabla u)=f(s,\mx)\\
u\in {\rm H}^1_0(\Omega)\,,\\
\end{array}
\right.
\end{equation}
where $\Omega = [0,1]^2$ is the square. In the context of heat conduction, $u$ denotes the temperature field, $f$ represents the density of the external heat source, and the objective is to maximize the total heat energy  $\int_\Omega fu\,\mathrm{d}\mx$. %
Therefore, in our risk-averse formulation, we seek to minimize large deviations of
\[
J(\theta,\bA)=-\int_\Omega u(s,\mx)f(s,\mx)\,\mathrm{d}\mx\,.
\]
The function $\varrho$ represents the constraint on the total amount of the first phase, given by 
\[\varrho(\theta)=\lambda\int_\Omega\theta,\]
where $\lambda $ is the associated Lagrange multiplier.

Note that the subdifferential of $\cR = \cvarg$ can be characterized as follows. For a given $X\in\pL1S$, the subdifferential $\partial \cvarg(X)$ is a singleton containing the unique element $\pi\in\Linf S$ defined by \cite{SDR14}
\[
\pi(s)=\begin{cases}
    0,& X(s)<\varg(X)\\
    \frac1\gamma,& X(s)>\varg(X)
\end{cases}
\]
(bearing in mind that $\mathbb{P}(X=\varg(X))=0$).
To facilitate the evaluation of  $\varg(\cJ(\theta,\bA))$, we employ the Karhunen--Lo\`eve spectral decomposition, which separates the spatial and stochastic components:
\[
f(s,\mx)=f_0(\mx)+\sum_{j=1}^\infty \sqrt{\lambda_j}f_j(\mx)\xi_j(s)\,,
\]
where $\{\xi_i\}_{i\in\N}$ are mutually uncorrelated random variables with mean zero and unit variance.
We focus on the  truncated sum 
\[
f(s,\mx)=f_0(\mx)+\sum_{j=1}^N \sqrt{\lambda_j}f_j(\mx)\xi_j(s)\,.
\]
If $u_j$ denotes the solution of the deterministic state equation \eqref{sdif} with the right-hand side $f_j$, for $j=0, \ldots, N$, then the perturbed state function is  simply given by
\[
u(s,\mx)=u_0(\mx)+\sum_{j=1}^N \sqrt{\lambda_j}u_j(\mx)\xi_j(s)\,,
\]
which leads to  
\[
\cJ(\theta,\bA)(s)=c_0+\sum_{j=1}^Nc_j\xi_j(s)+\sum_{i=1}^N\sum_{j=i}^Nc_{ij}\xi_i(s)\xi_j(s)
\]
where
\[
\begin{aligned}
c_0&=-\int_\Omega u_0(\mx)f_0(\mx)\,\mathrm{d}\mx,\\
c_j&=-\sqrt{\lambda_j}\int_\Omega u_0(\mx)f_j(\mx)+u_j(\mx)f_0(\mx)\,\mathrm{d}\mx\,,\; j=1,\ldots, N,\\
c_{ij}&=-\sqrt{\lambda_i\lambda_j}\int_\Omega u_i(\mx)f_j(\mx)+u_j(\mx)f_i(\mx)\,\mathrm{d}\mx\,,\; i,j=1,\ldots, N,\; j\geq i\,.
\end{aligned}
\]
For the optimality criteria method, let us denote $\varg(\cJ(\theta^k,\bA^k))$ simply by $\delta_k$, and introduce the set
\[
\Delta_k=\{s\in S:\cJ(\theta^k,\bA^k)(s)\geq \delta_k\}\,,
\]
so that the subgradient $\pi^k$ is given by $\frac1\gamma\chi_{\Delta_k}$.

In the problem under consideration, we have $g_\alpha=g_\beta$, which implies that the first integral in \eqref{eq:qk_alg} vanishes. Furthermore, the problem is self-adjoint; more precisely $p^k=-u^k$, so that
\[
\bM^k=-\nabla u_0^k \otimes \nabla u_0^k-2\sum_{j=1}^N d_j^k\Sym \nabla u_0^k \otimes \nabla u_j^k
-2\sum_{i=1}^N\sum_{j=i}^Nd_{ij}^k\Sym \nabla u_i^k \otimes \nabla u_j^k
\]
where
\[
\begin{aligned}
d_j^k&=\frac1\gamma\sqrt{\lambda_j}\int_{\Delta_k} \xi_j\,\mathrm{d}\mathbb{P}(s)\,,\; j=1,\ldots, N,\\
d_{ij}^k&=\frac1\gamma\sqrt{\lambda_i\lambda_j}\int_{\Delta_k} \xi_i\xi_j\,\mathrm{d}\mathbb{P}(s)\,,\; i,j=1,\ldots, N, j\geq i\,.
\end{aligned}
\]
We employ an expansion based on an exponential covariance function. We begin with the zero-mean random field
 $g(s,\mx)=f(s,\mx)-f_0(\mx)$, 
since $\mathbb{E}[f(\cdot,\mx)]=f_0(\mx)$. %
The covariance function is given by
\[
C(\mx,\mx')=\mathbb{E}[g(\cdot,\mx)g(\cdot,\mx')],
\]
and in this example we specify $C$ to be exponential covariance function:
\[
C(\mx,\mx')=\sigma^2 \exp \left({\displaystyle-\sum_{k=1}^d\frac{|x_k-x_k'|}{\eta_k}}\right)\,,
\]
where $\sigma^2$ and $\eta_k$ are the variance and the correlation lengths.
For this choice of covariance, the eigenpairs $(\lambda_j,f_j)$ of the associated integral operator
  $\cC\colon\pL2\Omega\to\pL2\Omega$ defined by
\[
\cC\psi(\mx)=\int_\Omega C(\mx,\mx')\psi(\mx')\,\mathrm{d}\mx'
\]
can be computed explicitly \cite{ZL04}. 
Moreover, the random coefficients $\xi_j$ in the resulting expansion are standard normally distributed.

The numerical example considers the two-dimensional setting $\Omega=[0,1]^2$ (as mentioned before), with a constant mean value $f_0=1$ on $\Omega$.
We choose $N=5$, $\sigma=2$, level $\gamma=0.9$, and examine several values of the correlation length $\eta:=\eta_1=\eta_2$.
The algorithm is implemented in {\rm \texttt{Freefem++}}, with P1 elements.
At each iteration of the optimality criteria method,  the $\gamma$-quantile $\varg(\cJ(\theta^k,\bA^k))$ and the integrals appearing in the definition of $\bM^k$ are approximated using a Monte Carlo procedure with a sample size of 10,000.

As is typical of the optimality criteria method, only a few iterations are required to obtain a good approximation of the solution, and the result is largely insensitive to the choice of the initial design.  \cref{fig:res} presents the designs obtained after 10 iterations for three values of the parameter $\eta$, alongside the design corresponding to the unperturbed right-hand side $f_0=1$ for comparison. 

\begin{figure}[H]
\centering
\includegraphics[scale=0.45]{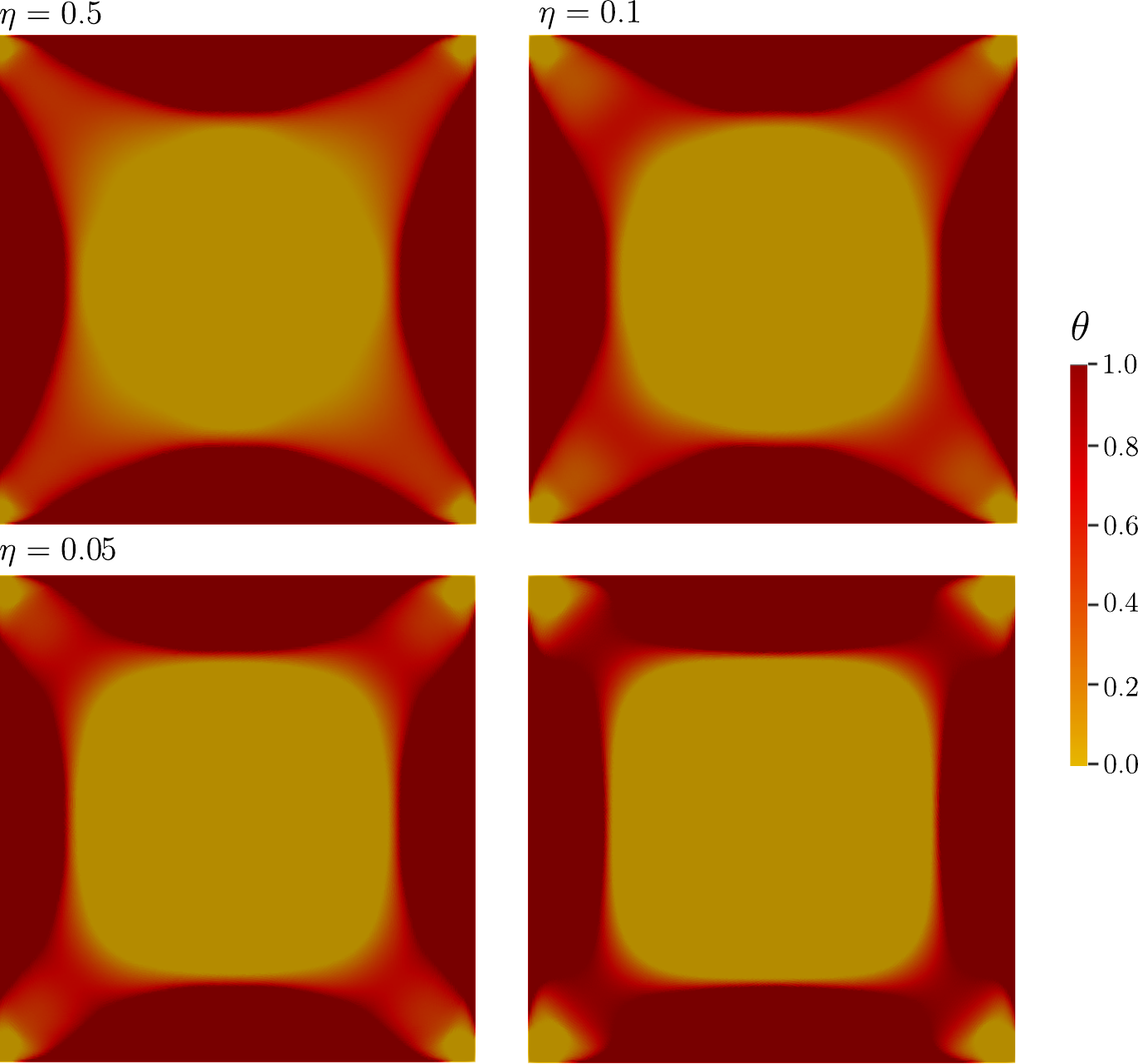}
\caption{Numerical solutions for 
$\eta=0.5$, $0.1$, and $0.05$ (displayed left to right, top to bottom) for $\cR = \cvarg$, 
along with the solution corresponding to the unperturbed right-hand side.}\label{fig:res}
\end{figure}
We see that as $\eta$ shrinks (i.e. as the influence of randomness reduces), the optimal shapes resemble more and more the deterministic shape seen in the bottom right corner of \cref{fig:res}.

 For purposes of comparison, we also consider the risk-neutral case in which the risk measure coincides with the expectation $\cR = \mathbb{E}$. The corresponding designs are shown in \cref{fig:alpha0}, again for various values of~$\eta$ and the unperturbed right-hand side in the rightmost panel. It can be observed that the obtained solutions are close to those of the corresponding non-perturbed problem. This behavior can be explained by the fact that the random fluctuations have zero mean, and thus tend to cancel out in expectation. 
 
 In contrast, for the risk-averse formulation, extreme realizations dominate the objective, and the optimizer explicitly accounts for adverse events. As a result, the solution becomes more conservative, prioritizing robustness over average performance. This approach reduces the likelihood of poor outcomes under unfavorable realizations, even at the expense of slightly suboptimal performance under typical conditions.

\begin{figure}[H]
    \centering
    \includegraphics[width=0.97\linewidth]{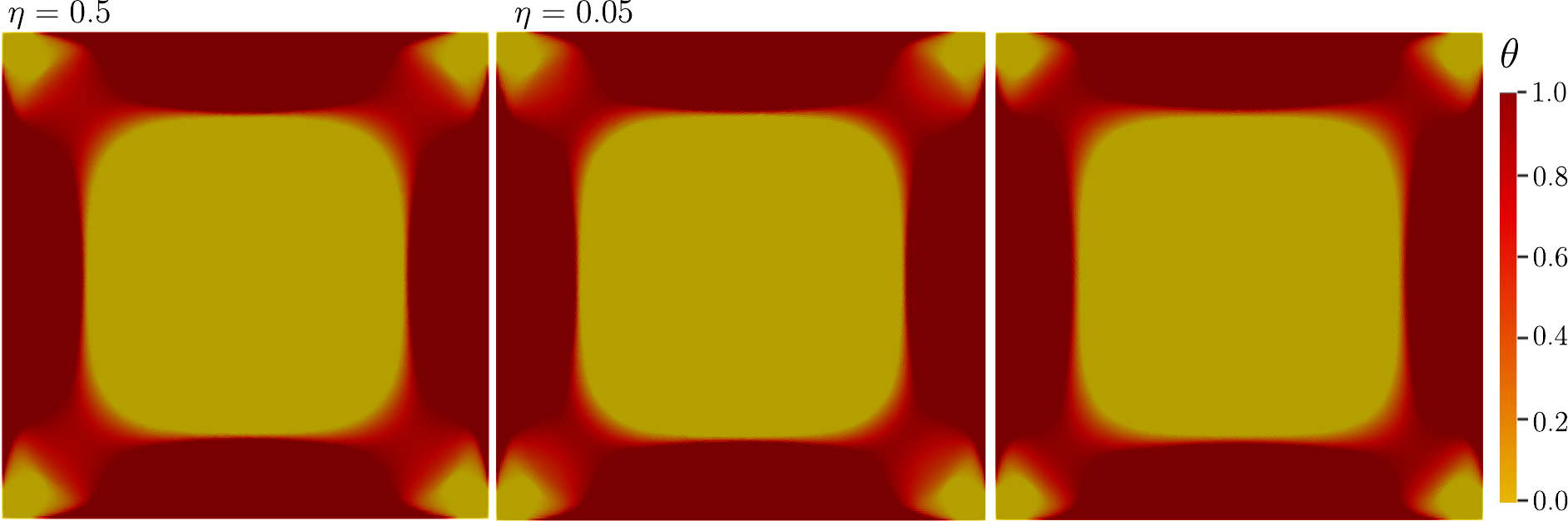}
    \caption{(Risk-neutral case) Numerical solutions for 
$\eta=0.5$, $0.05$ (displayed left to right) for $\cR = \mathbb{E}$, 
along with the solution corresponding to the unperturbed right-hand side.}
    \label{fig:alpha0}
\end{figure}

\section{Conclusion}\label{sec:conclusion}
In this paper, we studied a risk-averse stochastic optimization problem involving optimal designs. We proved existence and derived stationarity conditions, and used those to set up a numerical scheme and numerically solved an example. Our work can be seen as a first step in taking into account risk measures to deal with uncertainties in optimal design problems in conductivity. Further work could involve a more sophisticated analysis specialized to certain problem classes 
in applications and the resulting numerics and/or the development of tailored solution algorithms.

\section*{Acknowledgements}
The research of P.~Kunštek and M.~Vrdoljak  has been supported in part by Croatian Science Foundation under the project IP-2022-10-7261. 
The calculations were performed at the Laboratory for Advanced Computing (Faculty of Science, University of Zagreb). A. Alphonse thanks René Henrion for very helpful discussions.

\nocite{}
\bibliographystyle{abbrv}
\bibliography{optimal_design}

\end{document}